\documentclass [11pt,twoside,a4paper]{article}
\usepackage{mathrsfs}
\usepackage{amssymb}
\usepackage{amsthm}
\usepackage{amsfonts}
\usepackage{amsmath}
\usepackage{amstext}
\usepackage{amscd}
\usepackage{dsfont}
\newtheorem{theorem}{Theorem}[section]
\newtheorem{proposition}{Proposition}[section]

\newtheorem{definition}{Definition}[section]

\def\qdl{\textquotedblleft}

\def\BE{\begin{equation}}
\def\EE{\end{equation}}

\newcommand{\bey}{\begin{eqnarray}}
\newcommand{\eey}{\end{eqnarray}}
\newcommand{\beyy}{\begin{eqnarray*}}
\newcommand{\eeyy}{\end{eqnarray*}}

\begin{document}

\title{$\ \ $ Classification of non-CSC extremal K\"{a}hler metrics on K-surfaces $S^2_{\{\alpha\}}$ and $S^2_{\{\alpha,\beta\}}$ }
\author{ Yingjie Meng, Zhiqiang Wei }

\date{}
\maketitle

\begin{abstract}
 We commonly refer to  an extremal K\"{a}hler metric with finitely many singularities on a compact
 Riemann surface as an HCMU (Hessian of the Curvature of the Metric is Umbilical) metric. In this study, we specifically classify  non-CSC HCMU metrics on the K-surfaces $S^2_{\{\alpha\}}$ and $S^2_{\{\alpha,\beta\}}$.\par
 \vspace*{2mm}


\vspace*{2mm}
\noindent{\bf 2020 MR Subject Classification:}\hskip3mm 58E11, 53C55.

\end{abstract}

\section{Introduction}
 As is well known, within a given K\"{a}hler class of a compact K\"{a}hler manifold $\mathcal{M}$, an extremal K\"{a}hler metric, introduced by E. Calabi in \cite{Ca}, is the critical point of the following Calabi energy functional
$$
\mathcal{C}(g)=\int_{\mathcal{M}} R^2 dg,
$$
where $R$ denotes the scalar curvature of the metric $g$ in the K\"{a}hler class. The objective is to determine the \qdl best" metric within a fixed K\"{a}hler class. The Euler-Lagrange equations for the functional $\mathcal{C}(g)$ are given by $R_{,\alpha\beta}=0$ for all indices $\alpha, \beta$, with $R_{,\alpha\beta}$ being $(2,0)$ part of the second covariant derivative of $R$. When $\mathcal{M}$ is a compact Riemann surface, Calabi \cite{Ca} proved that an extremal K\"{a}hler metric is of constant scalar curvature (CSC) metric. This coincides with the classical uniformization theorem, which asserts that every Riemann surface admits a CSC metric.\par

A natural question then arises: do extremal K\"{a}hler metrics with singularities on a compact Riemann surface still possess a CSC property? This inquiry represents an attempt to generalize the classical uniformization theorem to a K-surface. The existence or nonexistence of CSC metrics on surfaces with conical singularities has been extensively studied in works such as \cite{CWWX,E1,E2,LSX,T,LT,UY} and further references cited therein. In \cite{Ch99},  X.X. Chen made a significant breakthrough by presenting an example of a non-CSC extremal K\"{a}hler metric with singularities. He also classified all extremal K\"{a}hler metrics on compact Riemann surfaces possessing finite cusp singularities and having finite area and energy.
\begin{theorem}[\cite{Ch99}]\label{Chen99}
Let $\mathcal{M}$ be a compact Riemann surface, $g$ an extremal K\"{a}hler metric with finite energy and area on $\mathcal{M}\setminus\{p_{j}\}_{j=1,\ldots,n}$, and $K$ the Gaussian curvature of $g$. Suppose that all singularities $p_{j}$ are weak cusps. Then the following classification holds.\par
(i) If genus$(\mathcal{M})\geq1$, then $K\equiv Const$.\par
(ii) If $\mathcal{M}=S^{2}$ and $n\geq 3$, then $K\equiv Const$.\par
(iii) If $\mathcal{M}=S^{2}$ and $n=2$, then there is no extremal K\"{a}hler metric.\par
(iv) If $\mathcal{M}=S^{2}$ and $n=1$, then $g$ is a rotationally symmetric metric that is determined uniquely by the total area.\par
In particular, $g$ is a metric with cusps.
\end{theorem}
In \cite{WZ}, G.F. Wang and X.H. Zhu discovered that every singular point of an extremal metric with finite energy and area on a singular surface is either a weak cusp or a conical singularity. They then generalized \textbf{Theorem} \ref{Chen99} as follows.
\begin{theorem}[\cite{WZ}]\label{W-Z}
Let $\mathcal{M}$ be a compact Riemann surface, $g$ an extremal K\"{a}hler metric with finite energy and area on $\mathcal{M}\setminus\{p_{j}\}_{j=1,\ldots,n}$, and $K$ the Gaussian curvature of $g$. Then $g$ is a conical metric with singular angle $2\pi\alpha_{j}~(j=1,\ldots,n)$(which may include some weak cusps). Furthermore, if all singular angles satisfy $2\pi\alpha_{j}<\pi$, then the following classification holds.\par
(i) If genus$(\mathcal{M})\geq1$, then $K\equiv Const$.\par
(ii) If $\mathcal{M}=S^{2}$ and $n\geq 3$, then $K\equiv Const$.\par
(iii) If $\mathcal{M}=S^{2}$ and $n=2$, there are two cases.\par
~~~(a) If both singular points are cusp, then there is no extremal K\"{a}hler metric.\par
~~~(b) If one of the singular points is not a cusp, then $g$ is a rotationally symmetric extremal K\"{a}hler metric that is determined uniquely by the total area and two angles $2\pi\alpha_{j}$.\par
(iv) If $\mathcal{M}=S^{2}$ and $n=1$, then $g$ is a rotationally symmetric metric that is determined uniquely by the total area and angle $2\pi\alpha$.\par
\end{theorem}
Now, we commonly refer to an extremal K\"{a}hler metric with finitely many singularities on a
compact Riemann surface as an HCMU (Hessian of the Curvature of the Metric is Umbilical) metric.  In \cite{Ch00}, X.X. Chen presented a theorem, famously known as the obstruction theorem for non-CSC HCMU metrics with conical singularities.

\begin{theorem}[\cite{Ch00}]\label{Obstruction}
Let $g$ be a non-CSC HCMU metric in a K-surface $\mathcal{M}_{\{\alpha_{1},\ldots,\alpha_{N}\}}$. Then the Euler character of the underlying surface should be determined by
$$\chi(\mathcal{M})=\sum_{j=1}^{J}(1-\alpha_{j})+(N-J)+S$$
where $2\pi\alpha_{1},\ldots,2\pi\alpha_{J}$ are all the singular angles corresponding to the saddle points of the Gaussian curvature $K$ of $g$, and $S$ is the number of smooth critical points of $K$.
\end{theorem}

 According to \textbf{Theorem} \ref{Obstruction}, if the curvature  $K$ of a non-CSC HCMU metric has no saddle points, the underlying surface must be a sphere, and $K$ possesses exactly two extremal points. Such an HCMU metric exhibits rotational symmetric, which is called a football. Furthermore, Chen's classification in \cite{Ch00} details the properties of football metrics.
\begin{theorem}[\cite{Ch00}]\label{Chen00}
If none of the prescribed angles in a K-surface is an integer multiple of $2\pi$, then this K-surface supports a non-CSC HCMU metric if and only if it is a football with two different angles at the two poles. Furthermore, any non-CSC HCMU metric in these footballs must be rotationally symmetric and uniquely determined by the area.
\end{theorem}

In \cite{LZ}, C.S. Lin and X.H. Zhu introduced a class of non-CSC HCMU metrics on $S^{2}$ with finitely many conical singularities of angles $2\pi\cdot integers$. These special non-CSC HCMU metrics are called exceptional when all singularities are saddle points of the Gaussian curvature $K$. A minimal exceptional HCMU metric is one where $K$ has a single  minimum. They provided an explicit formula for such metrics, showing they are determined by 3 parameters. In \cite{CW1}, Q. Chen and Y.Y. Wu derived an explicit formula for non-CSC HCMU metrics on $S^{2}$ and $T^{2}$, generalizing the results from \cite{LZ}.  In \cite{CCW}, Q. Chen, X.X. Chen and Y.Y. Wu proved that non-CSC HCMU metrics are locally isometric to a football, and reduced the existence question to a combinatorial problem. For $\mathcal{M}=S^{2}$, they showed  that  Chen's obstruction theorem is sufficient for the existence of  non-CSC HCMU metrics. In \cite{CW2,Xb}, the authors  further reduced the existence problem to a meromorphic 1-form on the underlying Riemann surface. However, the existence of  such a form remains challenging. On $S^2$, this problem can be translated into an algebraic one, which led to a class of non-CSC HCMU metrics with conical singularities whose existence is independent of singularity positions, as demonstrated in \cite{Wei}. In \cite{Wei19}, Y.Y. Wu and the second author  classified normalized non-CSC HCMU metrics on the K-surface $S^{2}_{\{2,2,2\}}$. Normalization in this context means that the Gaussian curvature $K$ of the metric is constrained such that its maximum is equal to $1$. \par

As far as we know, the classification of non-CSC HCMU metrics on $S^{2}$ with conical singularities \textemdash even those with a small number \textemdash remains unexplored. Therefore, this paper focuses on the classification problem for non-CSC HCMU metrics with finite area and energy on K-surfaces $S^2_{\{\alpha\}}$ and $S^2_{\{\alpha,\beta\}}$. Our results are as follows.
\begin{theorem}\label{mainth1}
  Suppose $g$ is a non-CSC HCMU metric on $S^{2}_{\{\alpha\}}$ and $K$ is the Gaussian curvature of $g$, then the following classification holds.\\
 (1) If the singularity is an extremal points of $K$, then  $g$ is a rotationally symmetric metric that is uniquely determined  by the total area and the angle $2\pi\alpha$. Explicitly, if $0\leq\alpha<1$, the singularity is the minimum point of $K$; if $\alpha>1$, the singularity is the maximum point of $K$.\\
 (2) If the  singularity is the saddle point of $K$, then $2\leq\alpha\in \mathbb{Z}$ and $K$ has $I_{1}\geq1$ maximum points and $I_{2}\geq1$ minimum points, where $I_{1}$ and $I_{2}$ satisfy $I_{1}>I_{2},I_{1}+I_{2}=\alpha+1$ and one of the following conditions holds.\par
 (2-1) $I_{2}=1$.\par
 (2-2) $I_{2}\geq2$ and $I_{2}\nmid I_{1}$.
\end{theorem}

\begin{theorem}\label{mainth2}
  Suppose $g$ is a non-CSC HCMU metric on $S^{2}_{\{\alpha,\beta\}}$ and $K$ is the Gaussian curvature of $g$, then the following classification holds.\\
 (1) If two singularities are extremal points of $K$, then $g$ is a rotationally symmetric metric that is uniquely determined  by the total area and the angles $2\pi\alpha,2\pi\beta$, i.e., g is a football.\\
 (2) If one  singularity is the saddle point of $K$ and  the other is an extremal point of $K$, without loss of generality, suppose the singularity of conical angle $2\pi\alpha ~(2\leq\alpha\in \mathbb{Z})$ is the saddle point of $K$, then the following statements hold.\par
 (A) If $2\leq\beta\in\mathbb{Z}$, there are two cases.\par
 (A-1) If the singularity of singular angle $2\pi\beta$ is a maximum point of $K$, then $K$ has $I_{1}+1\geq1$ maximum points and $I_{2}\geq1$ minimum points, where $I_{1}$ and $I_{2}$ satisfy $I_{1}+I_{2}=\alpha$ and one of the following conditions holds.\par
 ~~(a-1-1) $I_{2}=1$.\par
 ~~(a-1-2) $I_{2}\geq2,I_{2}\mid (I_{1}+\beta)$ and $I_{1}+\beta>\alpha-1$.\par
 ~~(a-1-3) $I_{2}\geq2$ and $I_{2}\nmid (I_{1}+\beta)$.\par
 (A-2) If the singularity of singular angle $2\pi\beta$ is a minimum point of $K$, then $K$ has $I_{1}\geq1$ maximum points and $I_{2}+1\geq1$ minimum points,  where $I_{1}$ and $I_{2}$ satisfy $I_{1}+I_{2}=\alpha$ and one of the following conditions holds.\par
 ~~(a-2-1) $I_{2}=0$.\par
 ~~(a-2-2) $I_{2}\geq1,(I_{2}+\beta)\mid I_{1}$ and $I_{1}(I_{2}+\beta)\geq (\alpha-1) {\rm GCD}(I_{2}+\beta,I_{1})$.\par
  (B) If $\beta\notin\mathbb{Z}$, there are two cases.\par
   (B-1) If the singularity of singular angle $2\pi\beta$ is a maximum point of $K$, then $K$ has $I_{1}+1\geq1$ maximum points and $I_{2}\geq1$ minimum points,  where $I_{1}$ and $I_{2}$ satisfy $I_{1}+I_{2}=\alpha$ and $I_{1}+\beta>I_{2}$.\par
   (B-2) If the singularity of singular angle $2\pi\beta$ is a minimum point of $K$, then $K$ has $I_{1}\geq1$ maximum points and $I_{2}+1\geq1$ minimum points,  where $I_{1}$ and $I_{2}$ satisfy $I_{1}+I_{2}=\alpha$ and $I_{1}>I_{2}+\beta$.\\
 (3) If both singularities are saddle points of $K$, then $2\leq\alpha,\beta\in\mathbb{Z}$ and $K$ has $I_{1}\geq1$ maximum points and $I_{2}\geq1$ minimum points, where $I_{1}$ and $I_{2}$ satisfy $I_{1}>I_{2},I_{1}+I_{2}=\alpha+\beta$ and one of the following conditions holds.\par
 ~~(3-1) $I_{2}=1$.\par
 ~~(3-2) $I_{2}\geq2,I_{2}\mid I_{1}$ and $I_{1}\geq max\{\alpha,\beta\}$.\par
 ~~(3-3) $I_{2}\geq2$ and $I_{2}\nmid I_{1}$.
\end{theorem}

\section{Preliminaries}

\subsection{HCMU metrics}

\begin{definition}[\cite{WZ}]
 (1) Let $g=e^{2\psi}|dz|^{2}$ be an extremal metric on punctured disk $\mathbb{D}\setminus\{0\}$, where $\mathbb{D}=\{z\in \mathbb{C}:|z|<1\}$. The singular point $z=0$ is called a weak conical point with singular angle $2\pi\alpha$ if and only if $\psi$ satisfies
 \begin{equation}\label{Weakconic}
 \lim_{r\rightarrow0}\frac{1}{2\pi}\int_{0}^{2\pi}(r\frac{\partial \psi(r,\theta)}{\partial r}+1-\alpha)d\theta=0.
 \end{equation}
 If $\alpha=0$ in (\ref{Weakconic}), then the singular point $z=0$ is called a weak cusp.\par
 (2) If $\psi$ can be locally expressed as
 \begin{equation*}\label{Conic}
 \psi(z)=(\alpha-1)\ln |z|+\rho(z)
 \end{equation*}
 with $\rho(z)$ a smooth function on $\mathbb{D}$ and $\alpha>0$, then the singular point $z=0$ is called a conical point with singular angle $2\pi\alpha$.\par
 (3) If $\psi$ can be locally expressed as
 \begin{equation*}\label{Cusp}
 \psi(z)=-\ln |z|+\ln\rho(z)
 \end{equation*}
 with $\rho(z)$ a smooth positive function on $\mathbb{D}$, then the singular point $z=0$ is called a cusp point.
\end{definition}
\begin{definition}[\cite{Ch00}]
 Let $\mathcal{M}$ be a compact Riemann surface and $p_1,\cdots,p_N$ be $N$ points on $\mathcal{M}$.
 Denote $\mathcal{M}\backslash \{p_1,\ldots,p_N\}$ by $\mathcal{M}^*$. Let $g$ be a conformal metric on $\mathcal{M}^*$.
 If $g$ satisfies
\begin{equation}\label{HCMUequ}
  K_{,zz}=0,
\end{equation}
where $K$ is the Gaussian curvature of $g$, then we call $g$ an HCMU (Hessian of the Curvature of the Metric is Umbilical) metric on $\mathcal{M}$.
\end{definition}
  In this paper, we restrict our attention to non-CSC HCMU metrics with finite area and
  finite Calabi energy, that is,
\begin{equation}\label{finite}
 \int_{\mathcal{M}^*}dg<+\infty, ~~ \int_{\mathcal{M}^*}
K^2 dg <+\infty.
\end{equation}
\par
From references \cite{Ch98,Ch00,WZ}, we know that each singularity of a non-CSC HCMU metric  is   conical or cusp if it has finite area and finite Calabi energy. Here, we summarize some key results of non-CSC HCMU metrics, which will be used in this paper. First the equation (\ref{HCMUequ}) is equivalent to
\begin{equation}
 \nabla K = \sqrt{-1}e^{-2\varphi}K_{\bar{z}}\frac{\partial}{\partial z}
\end{equation}
which is a holomorphic vector field on $\mathcal{M}^*$. Independently, authors in \cite{Ch00} and \cite{LZ} demonstrated that the curvature
$K$ can be  extended continuously to $\mathcal{M}$ and possesses a finite set of smooth extremal points on
$\mathcal{M}^*$. Further, \cite{CCW} and \cite{Xb} established that each smooth extremum of $K$ is either a global maximum (denoted $K_{1}$) or a global minimum (denoted $K_{2}$). Importantly, when all the singularities of $g$ are conical singularities,
$$
 K_1>0,~K_1>K_2>-(K_1+K_2);
$$
when there exist cusps in the singularities,
$$K_{1}>0,~K_{2}=-\frac{1}{2}K_{1}.$$
In \cite{LZ}, C.S. Lin and X.H. Zhu demonstrated that $\nabla K$ is actually a meromorphic vector field on $\mathcal{M}$.
In \cite{CW2,Xb}, the authors introduced the dual 1-form of $\nabla K$ by $\omega(\nabla K)=\frac{\sqrt{-1}}{4}$.
They call $\omega$ the character 1-form of the metric, which plays a significant role in studying non-CSC HCMU metrics. Now, denote $\mathcal{M}^* \setminus \{\text{smooth extremal points of}~K \}$ by $\mathcal{M}'$. Then on $\mathcal{M}'$
\begin{equation}\label{sys0}
\begin{cases}
 \cfrac{dK}{-\frac{1}{3}(K-K_1)(K-K_2)(K+K_1+K_2)}=\omega+\bar{\omega}, \\
g=-\frac{4}{3}(K-K_1)(K-K_2)(K+K_1+K_2)\omega \bar{\omega}.\\
\end{cases}
\end{equation}
From (\ref{sys0}), some properties of $\omega$ are derived in \cite{CW2} and \cite{Xb}:
\begin{itemize}
 \item All of the zeros of $\omega$  correspond to  conical singularities of $g$. For each zero of $\omega$ the corresponding singular angle is of the form $2\pi\alpha$ where $\alpha$ is an integer and the order of $\omega$ at the zero is $\alpha-1$. The function $K$ can be smoothly extended to these zeros, where $dK$ vanishes. At each zero of $\omega$, the value of $K$ falls within the range $[K_{1},K_{2}]$, hence we refer to these zeros as  saddle points of $K$.
 \item $\omega$ possesses only simple poles, which correspond to smooth extremal points of $K$ and singularities of $g$, excluding the zeros of $\omega$. Specifically, these poles of $\omega$ consist of all of the global maximum points and the minimum points of $K$. The residue of $\omega$ at each of these poles is a real number.
 \item   If all singularities of $g$ are conical singularities, denote the constant $-\frac{3}{(K_{1}-K_{2})(K_{2}+2K_{1})}$ by $\sigma$ and $-\frac{2K_{1}+K_{2}}{2K_{2}+K_{1}}$ by $\lambda$. Then at a maximum point of $K$ the residue of $\omega$ is $\sigma\alpha$ if at this point $g$ has the singular angle $2\pi\alpha$ or the residue of $\omega$ is $\sigma$ if this maximum point of $K$ is the smooth point of $g$. At a minimum point of $K$ the residue of $\omega$ is $\sigma\lambda\alpha$ if at this point $g$ has the singular angle $2\pi\alpha$ or the residue of $\omega$ is $\sigma\lambda$ if this minimum point of $K$ is the smooth point of $g$.
 \item If the singularities of $g$ contain conical and cusp singularities, then each cusp singularity of $g$ is a pole of $\omega$, the residue of $\omega$ at each cusp singularity is positive and each cusp singularity is a minimum point of $K$. Denote $K_{2}$ by $\mu$. Then if $p$ is a conical singularity with singular angle $2\pi\alpha$ and is a pole of $\omega$, the residue of $\omega$ at $p$ is $-\alpha/3\mu^{2}$ and $\lim_{x\rightarrow p}K(x)=-2\mu$; if $e$ is a smooth critical point of $K$, the residue of $\omega$ at $e$ is $-1/3\mu^{2}$ and $\lim_{x\rightarrow e}K(x)=-2\mu$.
  \item $\omega+\bar{\omega}$ is exact on $\mathcal{M}\setminus \{poles~of~\omega\}$.
 \end{itemize}

\subsection{Reduce the existence of non-CSC HCMU metrics to the existence of some kind of meromorphic 1-forms}
In this subsection, we will review some results in \cite{CW2} and \cite{CWWX}. For the convenience of the reader, we will provide detailed proofs of the main theorems. And we will review the energy integral formula for non-CSC HCMU metrics in \cite{WW16}. First, by virtue of a theorem from \cite{S}, we can establish the following theorem.
 \begin{theorem}[\cite{CW2}]\label{existenc of 1-form}
  Let $\mathcal{M}$ be a Riemann surface, $p_1,\ldots,p_L$ be $L(L\geq 2)$ points on $\mathcal{M}$ and $d_1,\ldots,d_L$ be $L$
  nonzero real numbers with $d_1+\ldots+d_L=0$. Then there exists a meromorphic 1-form $\omega$ on $\mathcal{M}$ such that
  \begin{itemize}
   \item[1)]$\omega$ only has $L$ simple poles at $p_1,\ldots,p_L$ with $Res_{p_l}(\omega)=d_l,~l=1,\ldots,L$,
   \item[2)]$\omega+\bar{\omega}$ is exact on $\mathcal{M} \setminus \{p_1,\ldots,p_L\}$.
  \end{itemize}
 \end{theorem}
\noindent Subsequently, Q. Chen and Y.Y. Wu \cite{CW2} established the following theorem.

\begin{theorem}[\cite{CW2}]\label{from 1-form to HCMU}
 Let $\mathcal{M}$ be a compact Riemann surface and $\omega$ be a meromorphic 1-form
 on $\mathcal{M}$ satisfying the conditions:
\begin{itemize}
  \item[1)]$\omega$ only has simple poles,
  \item[2)]At each pole the residue of $\omega$ is a real number,
  \item[3)]$\omega+\bar{\omega}$ is exact on $\mathcal{M}\setminus \{\text{poles of}~\omega\}$.
\end{itemize}
Then, there exists a non-CSC HCMU metric with  conical singularities such that $\omega$ is the character 1-form of the metric.
\end{theorem}
\begin{proof}
First by \textbf{Theorem} \ref{existenc of 1-form} a meromorphic 1-form on $\mathcal{M}$ satisfying the conditions 1), 2), 3) in \textbf{Theorem} \ref{from 1-form to HCMU} always exists. Suppose $p_{1},\ldots,p_{L}$ are the poles of $\omega$ in which $p_{1},\ldots,p_{J}$ are the poles where the residues of $\omega$ are negative and $p_{J+1},\ldots,p_{L}$ are the poles where the residues of $\omega$ are positive. Let $K_{1},K_{2}$ be two real numbers satisfying:
$$K_{1}>0,~~K_{1}>K_{2}>-(K_{1}+K_{2}).$$
 Consider the following equation:
\begin{equation}\label{K-equ-1}
 \cfrac{dK}{-\frac{1}{3}(K-K_1)(K-K_2)(K+K_1+K_2)}=\omega+\bar{\omega},~~\text{and}~K(p_{0})=K_{0},
\end{equation}
where $K_{2}< K_{0}< K_{1}$ and $p_{0}\in \mathcal{M}\setminus\{p_{1},\ldots,p_{L}\}$. One can prove that there exists a unique solution $K$ of (\ref{K-equ-1}) on $\mathcal{M}$ which satisfies that $K$ is smooth on $\mathcal{M}\setminus\{p_{1},\ldots,p_{L}\}$ and is continuous on $\mathcal{M}$. Then construct a metric
$$ g=-\frac{4}{3}(K-K_1)(K-K_2)(K+K_1+K_2)\omega \bar{\omega}.$$
One can prove $g$ is a non-CSC HCMU metric, $K$ is the Gauss curvature of $g$ with $K_{1},K_{2}$ being the maximum and the minimum of $K$ and $\omega$ is the character 1-form of $g$. Therefore $g$ has the conical singularities at the zeros and the poles of $\omega$. Denote $-\frac{3}{(K_{1}-K_{2})(K_{2}+2K_{1})}$ by $\sigma$ and $-\frac{2K_{1}+K_{2}}{2K_{2}+K_{1}}$ by $\lambda$. At the zeros of $\omega$ the singular angles of $g$ are of the form $2\pi(ord_{p}(\omega)+1)$, and at the poles of $\omega$ the singular angles of $g$ are of the form $2\pi\frac{Res_{p}(\omega)}{\sigma}$ or $2\pi\frac{Res_{p}(\omega)}{\lambda\sigma}$ depending on the sign of $Res_{p}(\omega)$. $\frac{Res_{p}(\omega)}{\sigma}=1$ or $\frac{Res_{p}(\omega)}{\lambda\sigma}=1$ means that $p$ is a smooth point of $g$.
\end{proof}

In \cite{Xb}, Q. Chen, B. Xu and Y.Y. Wu expanded upon \textbf{Theorem} \ref{from 1-form to HCMU} as follows.
\begin{theorem}[\cite{Xb}]\label{from 1-form to HCMU1}
 Let $\mathcal{M}$ be a compact Riemann surface and $\omega$ be a meromorphic 1-form
 on $\mathcal{M}$ satisfying the conditions:
\begin{itemize}
  \item[1)]$\omega$ only has simple poles,
  \item[2)]At each pole the residue of $\omega$ is a real number,
  \item[3)]$\omega+\bar{\omega}$ is exact on $\mathcal{M}\setminus \{\text{poles of}~\omega\}$.
\end{itemize}
Then, there exists a non-CSC HCMU metric with cusp singularities and conical singularities such that $\omega$ is the character 1-form of the metric.
\end{theorem}

\begin{proof}
Suppose $p_{1},\ldots,p_{L}$ are the poles of $\omega$ in which $p_{1},\ldots,p_{J}$ are the poles where the residues of $\omega$ are negative and $p_{J+1},\ldots,p_{L}$ are the poles where the residues of $\omega$ are positive. Let $\mu<0$ be a constant. Consider the following equation:
\begin{equation}\label{K-equ-2}
 \cfrac{dK}{-\frac{1}{3}(K-\mu)^{2}(K+2\mu)}=\omega+\bar{\omega},~~\text{and}~K(p_{0})=K_{0},
\end{equation}
where $\mu< K_{0}< -2\mu$ and $p_{0}\in \mathcal{M}\setminus\{p_{1},\ldots,p_{L}\}$. One can prove that there exists a unique solution $K$ of (\ref{K-equ-2}) on $\mathcal{M}$ which satisfies that $K$ is smooth on $M\setminus\{p_{1},\ldots,p_{L}\}$ and is continuous on $\mathcal{M}$. Then construct a metric
$$ g=-\frac{4}{3}(K-\mu)^{2}(K+2\mu)\omega \bar{\omega}.$$
One can prove $g$ is a non-CSC HCMU metric, $K$ is the Gauss curvature of $g$ with $-2\mu,\mu$ being the maximum and the minimum of $K$ and $\omega$ is the character 1-form of $g$. Therefore $g$ has the conical singularities at the zeros and the poles with negative residues of $\omega$, and cups singularities at the poles with positive residues of $\omega$. At the zeros of $\omega$ the singular angles of $g$ are of the form $2\pi(ord_{p}(\omega)+1)$, and at the poles with negative residues of $\omega$ the singular angles of $g$ are of the form $2\pi(-3\mu^{2}Res_{p}(\omega))$. $-3\mu^{2}Res_{p}(\omega)=1$ means that $p$ is a smooth point of $g$.
\end{proof}

To construct a  prescribed non-CSC HCMU metric, \textbf{Theorem} \ref{from 1-form to HCMU} or \ref{from 1-form to HCMU1} assert that it suffices to find a suitable meromorphic
1-form meeting the criteria. However, finding such a 1-form can be challenging due to the unknown  smooth points in the metric's singularities and the need to identify which points among the given ones are zeros of the 1-form. It's worth noting that  a meromorphic 1-form on the Riemann sphere $S^{2}$ that adheres to the conditions 1) and 2) in \textbf{Theorem} \ref{from 1-form to HCMU} or \ref{from 1-form to HCMU1} automatically fulfills condition 3) (cf. \cite{CW2}). By this fact, Y.Y. Wu and the second author \cite{Wei} established the following theorem.
\begin{theorem}[\cite{Wei}]\label{W-M}
Let $p_1,\ldots,p_N $ be $N (N \geq3)$ points on $S^{2}$ and
 $2\pi\alpha_{1},\ldots,2\pi\alpha_{N}$ be $N$
positive real numbers with $\alpha_{n}\neq1$, for $n = 1,2,\ldots,N$. If at least $ N-2$ of the $\alpha_{1},\ldots,\alpha_{N}$ are integers, then there exists a non-CSC HCMU metric which has conical singularities
$p_{1},\ldots,p_{N}$ with singular angles $2\pi\alpha_{1},\ldots,2\pi\alpha_{N}$ respectively.
\end{theorem}

At last, we introduce the energy integral formula for non-CSC HCMU metrics, which was proved by using Stokes' formula in \cite{WW16} and will be used in the proof of our classification.
\begin{theorem}[\cite{WW16}]\label{WW16}
Suppose $\mathcal{M}$ is a compact Riemann surface, $p_{1},\ldots,p_{N}$ are $N$ points on $\mathcal{M}$, and $g$ is a non-CSC HCMU metric on $\mathcal{M}$ with singular points $p_{1},\ldots,p_{N}$. Suppose $K$ is the Gaussian curvature of $g$, $\{q_{1},\ldots,q_{s}\}\subseteq \mathcal{M}\setminus\{p_{1},\ldots,p_{N}\}$ is the set of smooth extremal points of $K$, and $\omega$ is the character 1-form of $g$. Set $\mathcal{M}'=\mathcal{M}\setminus\{p_{1},\ldots,p_{N},q_{1},\ldots,q_{s}\}$, and
$$
\mathcal{C}_{n}(g)=\int_{\mathcal{M}'} K^{n} dg,~n=0,1,2,\ldots,
$$
which are called the $n$-th energy integral. Then
$$
\mathcal{C}_{n}(g)=\frac{6\alpha_{max}(K_{1}^{n+1}-K_{2}^{n+1})}{(n+1)(K_{1}-K_{2})(K_{2}+2K_{1})},
$$
where $\alpha_{max}$ is the sum of conical angles at the maximum points of $K$, and $K_{1},K_{2}$ are the maximum and minimum of $K$, respectively.
\end{theorem}

\subsection{One existence theorem for rational functions on the Riemann sphere}
In this subsection, we will review an existence result for rational functions on the Riemann sphere, as presented in \cite{SX20}. For more results,  we refer the reader to the references cited in \cite{SX20}.\par
Let $X$ and $Y$ be two compact, connected Riemann surfaces, and consider a holomorphic branched covering $f:X\rightarrow Y$ of degree $d$. At each point $q$ in $Y$, there is partition $\lambda(q)=(k_{1},\ldots,k_{r})$ of $d$ that characterizes the local behavior of $f$ near $q$. Over a suitable neighborhood of $q$ in $Y$, the map $f$ is equivalent to the map
$$\{1,\ldots,r\}\times \mathbb{D}\rightarrow \mathbb{D},~~(j,z)\mapsto z^{k_{j}},~~\text{where}~\mathbb{D}=\{z\in \mathbb{C}:|z|<1\},$$
with $q$ corresponding to $0$ in $\mathbb{D}$. For any partition $\lambda=(k_{1},\ldots,k_{r})$ of $d$, we define its length $Len(\lambda)=r$. We refer to a partition $\lambda$ of $d$ as non-trivial if $Len(\lambda)<d$. For the branched covering $f:X\rightarrow Y$, we call a point $q$ in $Y$ a branch point of $f$ if and only if $\lambda(q)$ is non-trivial, and we call the set of branch points of $f$ the branch set of $f$, denoted by $B_{f}$. The collection $\Lambda=\{\lambda(q):q\in B_{f}\}$ (with repetitions allowed) is called the branch data of $f$ and
$$v(f):=\sum_{q\in B_{f}}(d-Len(\lambda(q)))$$
the total branching data of $f$. By the Riemann-Hurwitz formula, we have
$$v(f)=2g(X)-2-d(2g(Y)-2),$$
where $g(X)$ (resp. $g(Y)$) denotes the genus of $X$ (resp. $Y$). Therefore, the total branching order $v(f)$ is an even non-negative integer.\par
A well-known realizability problem which arises in topology asks whether, given a compact connected Riemann surface $Y$ and a collection $\Lambda=\{\lambda_{1},\ldots,\lambda_{k}\}$ of non-trivial partitions of a positive integer $d$, there exist another compact connected Riemann surface $X$ together with a branched covering $f:X\rightarrow Y$ such that $\Lambda$ is its branch data. If such an $X$ and $f$ exist, we say that $\Lambda$ is realizable or realized by a branched covering. Boccara \cite{Bo82} obtained the following theorem.
\begin{theorem}[\cite{Bo82}]\label{M-F-T-3}
Suppose
$$\Lambda=\{(a_{1},\ldots, a_{p}), (b_{1},\ldots, b_{q}), (m+1,1,\ldots, 1)\}$$
is a collection of a partition of a positive integer $d$. Then there exists a branched covering of $\overline{\mathbb{C}}=\mathbb{C}\cup\{\infty\}$ with $\Lambda$ being its branch data if and only if it satisfies one of the following:\\
(i) $v(\Lambda)\geq 2d$ is even.\\
(ii) $v(\Lambda)=2d-2$ and $m<d/{\rm GCD}(a_{1},\ldots, a_{p},b_{1},\ldots, b_{q})$
\end{theorem}

Recently, J.J. Song and B. Xu \cite{SX20} generalized the second part of Boccara's result.
\begin{theorem}[\cite{SX20}]\label{M-F-T}
Let $d$ and $l$ be two positive integers. Consider
a collection
$$\Lambda=\{(a_{1},\ldots, a_{p}), (b_{1},\ldots, b_{q}), (m_{1}+1,1,\ldots, 1),\ldots, (m_{l}+1, 1,\ldots,1)\}$$
of $l+2$ partitions of $d$ where $(m_{1},\ldots,m_{l})$ is a partition of $p+q-2>0$.
Then there exists a rational function on $\overline{\mathbb{C}}=\mathbb{C}\cup\{\infty\}$ with $\Lambda$ being its branch data if and only if
$$max(m_{1},\ldots, m_{l})<d/{\rm GCD}(a_{1},\ldots, a_{p},b_{1},\ldots, b_{q}).$$
\end{theorem}

\section{Classification of non-CSC HCMU metrics on $S^{2}_{\{\alpha\}}$}
In this section, we will give the proof \textbf{Theorem} \ref{mainth1}. Our strategy involves a case-by-case analysis.\par
If $\alpha=0$, according to \textbf{Theorem} \ref{Chen99}, a non-CSC HCMU metric $g$ exists on $S^{2}_{\{\alpha\}}$ if and only if its single singularity is a cusp. Furthermore, such a metric is uniquely determined by its total area and is necessarily rotationally symmetric. Therefore, to classify non-CSC HCMU metrics on $S^{2}_{\{\alpha\}}$, we need only consider the case where the singularity is conical.\par

(A) For $2\leq\alpha\in \mathbb{Z}$, a non-CSC HCMU metric $g$ on $S^{2}_{\{\alpha\}}$ can exist under the following two conditions. \par

(A-1) The singularity is an extremal point of the Gaussian curvature $K$ of $g$.  \par
Since $\alpha>1$, the singularity is a maximum point of $K$.  Denote the maximum and minimum values of $K$ by $K_{1}$ and $K_{2}$ respectively. Set
 $$\sigma=-\frac{1}{(K_{1}-K_{2})(K_{2}+2K_{1})},~~~\lambda=-\frac{2K_{1}+K_{2}}{2K_{2}+K_{1}}.$$
 Suppose the character 1-form of $g$ is $\omega$. Since $g$ has a single singularity and the singularity is a maximum of $K$, $\omega$ has two simple poles. Regard $S^{2}$ as $\mathbb{C}\cup\{\infty\}$. We can assume that $0$ and $\infty$ are poles of $\omega$, with $\infty$ representing the conical singularity. Then $Res_{0}(\omega)=\sigma\lambda, Res_{\infty}(\omega)=\alpha\sigma$. By the Residue theorem, we obtain $\sigma\lambda+\alpha\sigma=\sigma(\lambda+\alpha)=0$, which implies $\alpha=-\lambda=\frac{2K_{1}+K_{2}}{2K_{2}+K_{1}}$. Thus
 $$\omega=-\frac{\sigma\alpha}{z}dz.$$
Conversely, by \textbf{Theorem} \ref{from 1-form to HCMU}, there  exists a non-CSC HCMU metric $g$ on $S^{2}_{\{\alpha\}}$ such that the singularity is the maximum point of the  Gaussian curvature $K$ of $g$. Furthermore, by \textbf{Theorem} \ref{WW16}, one can easily check that $g$ is uniquely determined by the total area and $\alpha$, and $g$ must be rotationally symmetric.

(A-2) The singularity is the saddle point of the Gaussian curvature $K$ of $g$.\par
 Denote the maximum and minimum values of $K$ by $K_{1}$ and $K_{2}$ respectively. Set
 $$\sigma=-\frac{1}{(K_{1}-K_{2})(K_{2}+2K_{1})},~~~\lambda=-\frac{2K_{1}+K_{2}}{2K_{2}+K_{1}}.$$
 Suppose the character 1-form of $g$ is $\omega$. Since $g$ has a single singularity and the singularity is the saddle point of $K$, $\omega$ has a simple zero of order $\alpha-1$ and $\alpha+1$ simple poles which are smooth extremal points of $K$.  Suppose $K$ has $I_{1}$ maximum points and $I_{2}$ minimum points, then $I_{1}+I_{2}=\alpha+1$. By $\lambda=-\frac{2K_{1}+K_{2}}{2K_{2}+K_{1}}< -1$, we obtain $I_{1}> I_{2}$. \par
 Regard $S^{2}$ as $\mathbb{C}\cup\{\infty\}$. We can assume that $0$ is the zero of $\omega$, and that $a_{1},\ldots,a_{I_{1}}$ are poles of $\omega$ with negative residues, while $b_{1},\ldots,b_{I_{2}}$ are poles of $\omega$ with positive residues. Then $Res_{a_{k}}(\omega)=\sigma,k=1,\ldots,I_{1}~ \text{and}~ Res_{b_{l}}(\omega)=\sigma\lambda,l=1,\ldots,I_{2}$. From these, we derive $I_{1}\sigma+I_{2}\sigma\lambda=\sigma(I_{1}+I_{2}\lambda)=0$, which leads to $\lambda=-\frac{I_{1}}{I_{2}}$. Thus there is a nonzero complex number $\widetilde{B}$ such that
 $$\omega=\frac{\widetilde{B}z^{\alpha-1}}{\prod_{k=1}^{I_{1}}(z-a_{k})\prod_{l=1}^{I_{2}}(z-b_{l})}dz=\sigma(\sum_{k=1}^{I_{1}}\frac{1}{z-a_{k}}-\frac{I_{1}}{I_{2}}\sum_{l=1}^{I_{2}}\frac{1}{z-b_{l}})dz.$$

\begin{proposition}\label{S-1-P-1}
Given an integer $\alpha\geq 2$, let $I_{1}$ and $I_{2}$ be two positive integers satisfying $I_{1}+I_{2}=\alpha+1$ and $I_{1}>I_{2}$. Then, there exists a meromorphic 1-form
$\omega$ on $S^{2}=\mathbb{C}\cup\{\infty\}$, defined by the form
 $$\omega=\frac{Bz^{\alpha-1}}{\prod_{k=1}^{I_{1}}(z-a_{k})\prod_{l=1}^{I_{2}}(z-b_{l})}dz=(\sum_{k=1}^{I_{1}}\frac{1}{z-a_{k}}-\frac{I_{1}}{I_{2}}\sum_{l=1}^{I_{2}}\frac{1}{z-b_{l}})dz,$$
where $B\in\mathbb{C}\setminus\{0\}$ is a constant, and $a_{1},\ldots,a_{I_{1}},b_{1},\ldots,b_{I_{2}}\in \mathbb{C}\setminus\{0\}$ are distinct complex numbers, if and only if one of the following two conditions holds.\\
(1) $I_{2}=1$;\\
(2) $I_{2}\geq 2$ and $I_{2}\nmid I_{1}$.
\end{proposition}

\begin{proof}
\textbf{Sufficiency}\par
If $I_{2}=1$, a direct calculation verifies the result.\par
For $I_{2}\geq2$ and $I_{2}\nmid I_{1}$, let $(I_{1},I_{2})=m\geq 1$. By \textbf{Theorem} \ref{M-F-T-3}, there exists a meromorphic function $f:\mathbb{C}\cup\{\infty\}\rightarrow \mathbb{C}\cup\{\infty\}$ of degree $\frac{I_{1}I_{2}}{m}$ and with branch data $$\Lambda=\{(\underbrace{\frac{I_{2}}{m},\ldots,\frac{I_{2}}{m}}_{I_{1}}),(\underbrace{\frac{I_{1}}{m},\ldots,\frac{I_{1}}{m}}_{I_{2}}),(\alpha,1,\ldots,1)\}.$$
Without loss of generality, we can express $f$ as
$$f(z)=\frac{C\prod_{k=1}^{I_{1}}(z-a_{k})^{\frac{I_{2}}{m}}}{\prod_{l=1}^{I_{2}}(z-b_{l})^{\frac{I_{1}}{m}}},$$
where $a_{1},\ldots,a_{I_{1}},b_{1},\ldots,b_{I_{2}}\in \mathbb{C}\setminus\{0\}$ are distinct complex numbers and $C\neq0$ is a constant.
Then
$$\frac{df}{f}=\frac{Bz^{\alpha-1}}{\prod_{k=1}^{I_{1}}(z-a_{k})\prod_{l=1}^{I_{2}}(z-b_{l})}dz=\frac{I_{2}}{m}(\sum_{k=1}^{I_{1}}\frac{1}{z-a_{k}}-\frac{I_{1}}{I_{2}}\sum_{l=1}^{I_{2}}\frac{1}{z-b_{l}})dz$$
where $B\neq0$ is a constant, is the desired meromorphic 1-form.\par
\textbf{Necessity}\par
Suppose there exists a meromorphic 1-form on $S^{2}=\mathbb{C}\cup\{\infty\}$, defined as follows
 $$\omega=\frac{Bz^{\alpha-1}}{\prod_{k=1}^{I_{1}}(z-a_{k})\prod_{l=1}^{I_{2}}(z-b_{l})}dz=(\sum_{k=1}^{I_{1}}\frac{1}{z-a_{k}}-\frac{I_{1}}{I_{2}}\sum_{l=1}^{I_{2}}\frac{1}{z-b_{l}})dz$$
where $B\in\mathbb{C}\setminus\{0\}$ is a constant and $a_{1},\ldots,a_{I_{1}},b_{1},\ldots,b_{I_{2}}\in \mathbb{C}\setminus\{0\}$ are distinct complex numbers.\par
If $I_{2}=1$, there is no further proof required. Consequently, we only need to consider the case (2).\par
Suppose $I_{2}\geq2$ and $I_{2}|I_{1}$, then the function
$$f(z)=exp(\int \omega)=C\frac{\prod_{k=1}^{I_{1}}(z-a_{k})}{\prod_{l=1}^{I_{2}}(z-b_{l})^{\frac{I_{1}}{I_{2}}}}$$
is a meromorphic function with degree $I_{1}$, where $C\neq0$ is a constant. The derivative of $f$ is
$$f'(z)=\frac{CBz^{\alpha-1}}{\prod_{l=1}^{I_{2}}(z-b_{l})^{\frac{I_{1}}{I_{2}}+1}}.$$
This implies $\alpha\leq I_{1}$. Consequently, $I_{1}=\alpha$ and $I_{2}=1$. It is a contradiction. So $I_{2}\nmid I_{1}$.
\end{proof}

 By \textbf{Proposition} \ref{S-1-P-1} and \textbf{Theorem} \ref{from 1-form to HCMU}, there exists a non-CSC HCMU metric $g$ on $S^{2}_{\{\alpha\}}$ such that the singularity is a saddle point of the Gaussian curvature $K$ of $g$. Furthermore, $K$ possesses $I_{1}$ maximum points and $I_{2}$ minimum points, where $I_{1}=\alpha,I_{2}=1$ or $I_{2}\geq2$ and $I_{2}\nmid I_{1}$, with $I_{1}> I_{2}$ and their sum satisfying $I_{1}+I_{2}=\alpha+1$. Since the number of saddle and extremal points of $K$ exceeds $3$, by virtue of \textbf{Theorem} \ref{WW16}, it is straightforward to ascertain that $g$ is uniquely determined by the total area, $\frac{I_{1}}{I_{2}}$, $\alpha$ and the initial value $K(z_{0})=K_{0}$. \par

(B) When $\alpha\notin \mathbb{Z}^{+}$, a non-CSC HCMU metric $g$ can exist on $S^{2}_{\{\alpha\}}$ under the following two conditions. \par
(B-1) The singularity is the maximum point of the Gaussian curvature $K$ of $g$. \par
First we obtain $\alpha>1$.  Denote the maximum and minimum values of $K$ by $K_{1}$ and $K_{2}$ respectively. Set
 $$\sigma=-\frac{1}{(K_{1}-K_{2})(K_{2}+2K_{1})},~~~\lambda=-\frac{2K_{1}+K_{2}}{2K_{2}+K_{1}}.$$
 Suppose the character 1-form of $g$ is $\omega$. Since $g$ has a single singularity and the singularity is a maximum of $K$, $\omega$ has two simple poles. Regard $S^{2}$ as $\mathbb{C}\cup\{\infty\}$. We can assume that $0$ and $\infty$ are simple poles of $\omega$, with $\infty$ being the conical singularity. Then $Res_{0}(\omega)=\sigma\lambda$ and $Res_{\infty}(\omega)=\alpha\sigma$. From this, we deduce that $\sigma\lambda+\alpha\sigma=\sigma(\lambda+\alpha)=0$, which implies $\alpha=-\lambda=\frac{2K_{1}+K_{2}}{2K_{2}+K_{1}}$. Thus
 $$\omega=-\frac{\sigma\alpha}{z}dz.$$
Conversely, by \textbf{Theorem} \ref{from 1-form to HCMU}, there  exists a non-CSC HCMU metric $g$ on $S^{2}_{\{\alpha\}}$ such that the singularity is the maximum point of the  Gaussian curvature $K$ of $g$. Furthermore, by \textbf{Theorem} \ref{WW16}, one can easily check that $g$ is uniquely determined by the total area and $\alpha$, and must be rotationally symmetric.

(B-2) The singularity is the minimum point of the Gaussian curvature $K$ of $g$. \par
First we obtain $\alpha<1$.  Denote the maximum and minimum values of $K$ by $K_{1}$ and $K_{2}$ respectively. Set
 $$\sigma=-\frac{1}{(K_{1}-K_{2})(K_{2}+2K_{1})},~~~\lambda=-\frac{2K_{1}+K_{2}}{2K_{2}+K_{1}}.$$
 Suppose the character 1-form of $g$ is $\omega$. Since $g$ has a single singularity and the singularity is a minimum of $K$, $\omega$ has two simple poles.  Regard $S^{2}$ as $\mathbb{C}\cup\{\infty\}$. We can assume that $0$ and $\infty$ are simple poles of $\omega$, with $\infty$ being the conical singularity. Then $Res_{0}(\omega)=\sigma$ and $Res_{\infty}(\omega)=\alpha\sigma\lambda$. From this, we deduce that $\sigma+\alpha\sigma\lambda=\sigma(1+\alpha\lambda)=0$, which implies $\alpha=-1/\lambda=\frac{2K_{2}+K_{1}}{2K_{1}+K_{2}}$. Thus
 $$\omega=-\frac{\sigma}{z}dz.$$
Conversely, by \textbf{Theorem} \ref{from 1-form to HCMU}, there  exists a non-CSC HCMU metric $g$ on $S^{2}_{\{\alpha\}}$ such that the singularity is the minimum point of the  Gaussian curvature $K$ of $g$. Furthermore, by \textbf{Theorem} \ref{WW16}, one can easily check that $g$ is uniquely determined by the total area and $\alpha$, and must be rotationally symmetric.

\section{Classification of non-CSC HCMU metrics on $S^{2}_{\{\alpha,\beta\}}$}
In this section, we will give the proof the proof \textbf{Theorem} \ref{mainth2}. Our strategy is the same as the proof of \textbf{Theorem} \ref{mainth1}.\par
By \textbf{Theorem} \ref{Chen99}, if $\alpha=\beta=0$, there is no non-CSC HCMU metric on $S^{2}_{\{\alpha,\beta\}}$. Consequently, at most one of the parameters $\alpha$ and $\beta$ can be equal to $0$.\par

Firstly, if there is a zero in $\alpha,\beta$, without loss of generality, we can suppose $\beta=0$. Then, a non-CSC HCMU metric $g$ on $S^{2}_{\{\alpha,\beta\}}$ exists if and only if $g$ possesses both a conical singularity and a cusp singularity. It's worth noting that G.F. Wang and X.H. Zhu have classified the case where $\alpha<\frac{1}{2}$, as stated in \textbf{Theorem} \ref{W-Z}. Now, we proceed to classify non-CSC HCMU metrics for all $0<\alpha\neq1$.\par
(A) If $2\leq\alpha\in \mathbb{Z}$, a non-CSC HCMU metric $g$ on $S^{2}_{\{\alpha,\beta\}}$ exists if and only if one of the following two cases holds.\par
 (A-1) The singularity of conical angle $2\pi\alpha$ is the maximum point of the Gaussian curvature $K$ of $g$.\par
Obviously, in this scenario, $K$ precisely possesses exactly two extremal points. Regard $S^{2}$ as $\mathbb{C}\cup\{\infty\}$. We can assume that the conical and cusp singularities of $g$ are $0$ and $\infty$ respectively. Denote the minimum of $K$ by $\mu<0$. Then the character 1-form of $g$ is $\omega=-\frac{\alpha}{3\mu^{2}}\cdot \frac{1}{z}dz$.\par
Conversely, by \textbf{Theorem} \ref{from 1-form to HCMU1}, there  exists a non-CSC HCMU metric $g$ on $S^{2}_{\{\alpha,\beta\}}$ such that the singularity of singular angle $2\pi\alpha$ is the maximum point of the Gaussian curvature $K$ of $g$ and another singularity is the minimum point of $K$. Furthermore, by \textbf{Theorem} \ref{WW16}, one can easily check that $g$ is uniquely determined by the total area and $\alpha$, and must be rotationally symmetric.\par
(A-2) The singularity of conical angle $2\pi\alpha$ is the saddle point of the Gaussian curvature $K$ of $g$.\par
In this case, $K$ possesses exactly $\alpha$ smooth maximum points. Regard $S^{2}$ as $\mathbb{C}\cup\{\infty\}$. We can assume that the conical and cusp singularities of $g$ are $0$ and $\infty$ respectively, and the smooth maximum points of $K$ are $a_{1},\ldots,a_{\alpha}$. Denote the minimum of $K$ by $\mu<0$. Then  the character 1-form of $g$ can be written as
$$\omega=-\frac{1}{3\mu^{2}}(\sum_{k=1}^{\alpha} \frac{1}{z-a_{k}})dz=\frac{\widetilde{B}z^{\alpha-1}}{\prod_{k=1}^{\alpha}(z-a_{k})}dz,$$
where $\widetilde{B}\neq0$ is a constant.\par
Conversely, it is straightforward to establish the following proposition.

\begin{proposition}\label{S-2-P-0}
Let $\alpha\geq2$ be an integer. For any real number $\mu<0$, there exists a meromorphic 1-form $\omega$ on $S^{2}=\mathbb{C}\cup\{\infty\}$, defined by the form
\begin{equation}
\omega=-\frac{1}{3\mu^{2}}(\sum_{k=1}^{\alpha} \frac{1}{z-a_{k}})dz=\frac{Bz^{\alpha-1}}{\prod_{k=1}^{\alpha}(z-a_{k})}dz,
\end{equation}
 where $B\in\mathbb{C}\setminus\{0\}$ is a constant and $a_{1},\ldots,a_{\alpha}\in \mathbb{C}\setminus\{0\}$ are distinct complex numbers.
\end{proposition}

 By \textbf{Proposition} \ref{S-2-P-0} and \textbf{Theorem} \ref{from 1-form to HCMU1}, there  exists a non-CSC HCMU metric $g$ on $S^{2}_{\{\alpha,\beta\}}$ such that the singularity of singular angle $2\pi\alpha$ is the saddle point of the Gaussian curvature $K$ of $g$ and another singularity is the minimum point of $K$. Additionally, by \textbf{Theorem} \ref{WW16}, one can easily check that $g$ is uniquely determined by the total area and $\alpha$, and the initial value $K(z_{0})=K_{0}$.\par
(B) If $\alpha\notin \mathbb{Z}$, a non-CSC HCMU metric $g$ on $S^{2}_{\{\alpha,\beta\}}$ exists if and only if the Gaussian curvature $K$ of $g$ has both a maximum and a minimum points.\par
 Regard $S^{2}$ as $\mathbb{C}\cup\{\infty\}$. We can assume that the conical and cusp singularities of $g$ are $0$ and $\infty$ respectively. Denote the minimum of $K$ by $\mu<0$. Then the character 1-form of $g$ is $\omega=-\frac{\alpha}{3\mu^{2}}\cdot \frac{1}{z}dz$.\par
Conversely, by \textbf{Theorem} \ref{from 1-form to HCMU1}, there  exists a non-CSC HCMU metric $g$ on $S^{2}_{\{\alpha,\beta\}}$ such that the singularity of singular angle $2\pi\alpha$ is the maximum point of the  Gaussian curvature $K$ of $g$ and another singularity is the minimum point of $K$. Furthermore, by \textbf{Theorem} \ref{WW16}, one can easily check that $g$ is uniquely determined by the total area and $\alpha$, and  must be rotationally symmetric.\par

Secondly, if $\alpha,\beta\notin \mathbb{Z}$, a non-CSC HCMU metric $g$ on $S^{2}_{\{\alpha,\beta\}}$ exists if and only if $\alpha\neq\beta$. Furthermore, $g$ must be a football. This classification is due to X.X. Chen,  as stated in \textbf{Theorem} \ref{Chen00}.\par

Thirdly, if there is only one positive integer in $\alpha,\beta$, without loss of generality, we can assume that $\alpha\in \mathbb{Z}^{+}$ and $0<\beta\notin \mathbb{Z}$. Then, a non-CSC HCMU metric $g$ on $S^{2}_{\{\alpha,\beta\}}$ exits if and only if  one of the following two conditions holds.\par

(A) If the singularity of conical angle $2\pi\alpha$ is an extremal point of the Gaussian curvature $K$ of $g$, it leads to two distinct scenarios. \par
(A-1) If the singularity of conical angle $2\pi\alpha$ is the maximum point of the Gaussian curvature $K$ of $g$,  then another singularity is the minimum point of $K$ and $\alpha>\beta$. Moreover, $g$ must be a football.\par
(A-2) If the singularity of conical angle $2\pi\alpha$ is the minimum point of the Gaussian curvature $K$ of $g$,  then another singularity is the maximum point of $K$ and $\alpha<\beta$. Moreover, $g$ must be a football.\par

(B) If the singularity of conical angle $2\pi\alpha$ is the saddle point of the Gaussian curvature $K$ of $g$, it leads to two distinct scenarios. \par

(B-1) If the singularity of conical angle $2\pi\beta$ is a maximum point of the Gaussian curvature $K$ of $g$,  then $K$ has exactly $\alpha$ smooth extremal points. Suppose $K$ has $I_{1}\geq0$ smooth maximum points and $I_{2}\geq1$ smooth minimum points, then $I_{1}+I_{2}=\alpha$.  Denote the maximum and minimum values of $K$ by $K_{1}$ and $K_{2}$ respectively. Set
 $$\sigma=-\frac{1}{(K_{1}-K_{2})(K_{2}+2K_{1})},~~~\lambda=-\frac{2K_{1}+K_{2}}{2K_{2}+K_{1}}.$$
 Suppose the character 1-form of $g$ is $\omega$. Since $g$ has two singularities and one is the saddle point of $K$ and another is a maximum point of $K$, $\omega$ has a zero of order $\alpha-1$, $\alpha$ simple poles which are smooth extremal points of $K$ and a simple pole which is the singularity of singular angle $2\pi\beta$.  Regard $S^{2}$ as $\mathbb{C}\cup\{\infty\}$. We can assume that $0$ is the zero of $\omega$, $a_{1},\ldots,a_{I_{1}},a_{I_{1}+1}$ are poles of $\omega$ at which residues are negative and  $b_{1},\ldots,b_{I_{2}}$ are poles of $\omega$ at which residues are positive. Then $Res_{a_{k}}(\omega)=\sigma,k=1,\ldots,I_{1},Res_{a_{I_{1}+1}}(\omega)=\sigma\beta~ \text{and}~ Res_{b_{l}}(\omega)=\sigma\lambda,l=1,\ldots,I_{2}$. From this, we derive $I_{1}\sigma+\sigma\beta+I_{2}\sigma\lambda=\sigma(I_{1}+\beta+I_{2}\lambda)=0$, which implies $\lambda=-\frac{I_{1}+\beta}{I_{2}}$. By $\lambda=-\frac{2K_{1}+K_{2}}{2K_{2}+K_{1}}=-\frac{I_{1}+\beta}{I_{2}}<-1$, we obtain $I_{1}+\beta> I_{2}$. Thus there is a nonzero $\widetilde{B}$ such that
 $$\omega=\frac{\widetilde{B}z^{\alpha-1}}{\prod_{k=1}^{I_{1}+1}(z-a_{k})\prod_{l=1}^{I_{2}}(z-b_{l})}dz=\sigma(\sum_{k=1}^{I_{1}}\frac{1}{z-a_{k}}+\frac{\beta}{z-a_{I_{1}+1}}-\frac{I_{1}+\beta}{I_{2}}\sum_{l=1}^{I_{2}}\frac{1}{z-b_{l}})dz.$$
Similar to the proof of \textbf{Proposition} \ref{S-2-P-2}, one can establish the following proposition.
\begin{proposition}\label{S-2-P-4}
Let $\alpha\geq 2,I_{1}\geq0$ and $I_{2}\geq1$ be integers such that $I_{1}+I_{2}=\alpha$. Additionally, let $0<\beta\notin\mathbb{Z}$ satisfy the condition $I_{1}+\beta>I_{2}$. Then there exists a meromorphic 1-form $\omega$ on $S^{2}=\mathbb{C}\cup\{\infty\}$, defined by the form
\begin{equation*}
\omega=\frac{Bz^{\alpha-1}}{(z-a_{I_{1}+1})\prod_{k=1}^{I_{1}}(z-a_{k})\prod_{l=1}^{I_{2}}(z-b_{l})}dz=(\sum_{k=1}^{I_{1}}\frac{1}{z-a_{k}}+\frac{\beta}{z-a_{I_{1}+1}}-\frac{I_{1}+\beta}{I_{2}}\sum_{l=1}^{I_{2}}\frac{1}{z-b_{l}})dz,
\end{equation*}
where $B\in\mathbb{C}\setminus\{0\}$ is a constant and $a_{1},\ldots,a_{I_{1}},a_{I_{1}+1},b_{1},\ldots,b_{I_{2}}\in \mathbb{C}\setminus\{0,1\}$ are distinct complex numbers.
\end{proposition}
By \textbf{Proposition} \ref{S-2-P-4} and \textbf{Theorem} \ref{from 1-form to HCMU}, there  exists a non-CSC HCMU metric $g$ on $S^{2}_{\{\alpha,\beta\}}$ such that the singularity of singular angle $2\pi\alpha$ is the saddle points of the  Gaussian curvature $K$ and the singularity of singular angle $2\pi\beta$ is a maximum point of $K$.\par
(B-2)  If the  singularity of conical angle $2\pi\beta$ is the minimum point of the Gaussian curvature $K$ of $g$,  then  $K$ has exactly $\alpha$ smooth extremal points. Suppose $K$ has $I_{1}\geq1$ smooth maximum points and $I_{2}\geq0$ smooth minimum points, then $I_{1}+I_{2}=\alpha$. Denote the maximum and minimum of $K$ by $K_{1}$ and $K_{2}$ respectively. Set
 $$\sigma=-\frac{1}{(K_{1}-K_{2})(K_{2}+2K_{1})},~~~\lambda=-\frac{2K_{1}+K_{2}}{2K_{2}+K_{1}}.$$
 Suppose the character 1-form of $g$ is $\omega$. Since $g$ has two singularities and one is the saddle point of $K$ and another is a maximum point of $K$, $\omega$ has a zero of order $\alpha-1$, $\alpha$ simple poles which are smooth extremal points of $K$ and a simple pole which is the singularity of singular angle $2\pi\beta$.  Regard $S^{2}$ as $\mathbb{C}\cup\{\infty\}$. We can assume that $0$ is the zero of $\omega$, $a_{1},\ldots,a_{I_{1}}$ are poles of $\omega$ at which residues are negative and $b_{1},\ldots,b_{I_{2}},b_{I_{2}+1}$ are poles of $\omega$ at which residues are positive. Then $Res_{a_{k}}(\omega)=\sigma,k=1,\ldots,I_{1}~ \text{and}~ Res_{b_{l}}(\omega)=\sigma\lambda,l=1,\ldots,I_{2},Res_{b_{I_{2}+1}}(\omega)=\sigma\lambda\beta$. From these, we derive $I_{1}\sigma+\sigma\lambda\beta+I_{2}\sigma\lambda=\sigma(I_{1}+\lambda\beta+I_{2}\lambda)=0$, which implies $\lambda=-\frac{I_{1}}{I_{2}+\beta}$. By $\lambda=-\frac{2K_{1}+K_{2}}{2K_{2}+K_{1}}=-\frac{I_{1}}{I_{2}+\beta}<-1$, we obtain $I_{1}> I_{2}+\beta$. Consequently, there exists a nonzero $\widetilde{B}$ such that
 $$\omega=\frac{\widetilde{B}z^{\alpha-1}}{\prod_{k=1}^{I_{1}}(z-a_{k})\prod_{l=1}^{I_{2}+1}(z-b_{l})}dz=\sigma(\sum_{k=1}^{I_{1}}\frac{1}{z-a_{k}}-\frac{\frac{I_{1}\beta}{I_{2}+\beta}}{z-b_{I_{2}+1}}-\frac{I_{1}}{I_{2}+\beta}\sum_{l=1}^{I_{2}}\frac{1}{z-b_{l}})dz.$$
Similar to the proof of \textbf{Proposition} \ref{S-2-P-3}, we can establish the following proposition.

\begin{proposition}\label{S-2-P-5}
Let $\alpha\geq 2,I_{1}\geq1$ and $I_{2}\geq0$ be integers such that $I_{1}+I_{2}=\alpha$. Additionally, let $0<\beta\notin\mathbb{Z}$ satisfy $I_{1}>I_{2}+\beta$. Then there exists a meromorphic 1-form $\omega$ on $S^{2}=\mathbb{C}\cup\{\infty\}$ defined by the form
\begin{equation*}
\omega=\frac{Bz^{\alpha-1}}{\prod_{k=1}^{I_{1}}(z-a_{k})\prod_{l=1}^{I_{2}+1}(z-b_{l})}dz=(\sum_{k=1}^{I_{1}}\frac{1}{z-a_{k}}-\frac{\frac{I_{1}\beta}{I_{2}+\beta}}{z-b_{I_{2}+1}}-\frac{I_{1}}{I_{2}+\beta}\sum_{l=1}^{I_{2}}\frac{1}{z-b_{l}})dz,
\end{equation*}
where $B\in\mathbb{C}\setminus\{0\}$ is a constant and $a_{1},\ldots,a_{I_{1}},a_{I_{1}},b_{1},\ldots,b_{I_{2}+1}\in \mathbb{C}\setminus\{0,1\}$ are distinct complex numbers.
\end{proposition}

By \textbf{Proposition} \ref{S-2-P-5} and \textbf{Theorem} \ref{from 1-form to HCMU}, there  exists a non-CSC HCMU metric $g$ on $S^{2}_{\{\alpha,\beta\}}$ such that the singularity of singular angle $2\pi\alpha$ is the saddle points of the  Gaussian curvature $K$ and the singularity of singular angle $2\pi\beta$ is a minimum point of $K$.\par

Fourthly, if $\alpha,\beta\in \mathbb{Z}^{+}$, a non-CSC HCMU metric $g$ on $S^{2}_{\{\alpha,\beta\}}$ exists if and only if one of the following two conditions holds.\par
(A) The two singularities are both saddle points of the Gaussian curvature $K$ of $g$. \par
 Denote the maximum and minimum of $K$ by $K_{1}$ and $K_{2}$ respectively. Set
 $$\sigma=-\frac{1}{(K_{1}-K_{2})(K_{2}+2K_{1})},~~~\lambda=-\frac{2K_{1}+K_{2}}{2K_{2}+K_{1}}.$$
 Suppose the character 1-form of $g$ is $\omega$. Since $g$ has exactly two singularities and the singularities are both saddle points of the Gauss curvature, $\omega$ has two zeros of orders $\alpha-1$ and $\beta-1$ respectively, and $\alpha+\beta$ simple poles which are smooth extremal points of $K$.  Suppose $K$ has $I_{1}$ maximum points and $I_{2}$ minimum points, then $I_{1}+I_{2}=\alpha+\beta$.  Regard $S^{2}$ as $\mathbb{C}\cup\{\infty\}$. We can assume that $0$ and $1$ are the zeros of $\omega$, $a_{1},\ldots,a_{I_{1}}$ are poles of $\omega$ at which residues are negative and  $b_{1},\ldots,b_{I_{2}}$ are poles of $\omega$ at which residues are positive. Then $Res_{a_{k}}(\omega)=\sigma,k=1,\ldots,I_{1}~ \text{and}~ Res_{b_{l}}(\omega)=\sigma\lambda,l=1,\ldots,I_{2}$. Applying the Residue theorem, we obtain $I_{1}\sigma+I_{2}\sigma\lambda=\sigma(I_{1}+I_{2}\lambda)=0$, which implies $\lambda=-\frac{I_{1}}{I_{2}}$. Since $\lambda=-\frac{2K_{1}+K_{2}}{2K_{2}+K_{1}}=-\frac{I_{1}}{I_{2}}<-1$, we obtain $I_{1}> I_{2}$. Consequently, there is a nonzero complex number $\widetilde{B}$ such that
 $$\omega=\frac{\widetilde{B}z^{\alpha-1}(z-1)^{\beta-1}}{\prod_{k=1}^{I_{1}}(z-a_{k})\prod_{l=1}^{I_{2}}(z-b_{l})}dz=\sigma(\sum_{k=1}^{I_{1}}\frac{1}{z-a_{k}}-\frac{I_{1}}{I_{2}}\sum_{l=1}^{I_{2}}\frac{1}{z-b_{l}})dz.$$
\begin{proposition}\label{S-2-P-1}
Let $\alpha\geq2,\beta\geq 2,I_{1}\geq2$ and $I_{2}\geq 1$ be 4 integers such that $I_{1}>I_{2}$ and $I_{1}+I_{2}=\alpha+\beta$. Then, there exists a meromorphic 1-form $\omega$ on $S^{2}=\mathbb{C}\cup\{\infty\}$ defined by the form
\begin{equation}\label{S-2-E0}
\omega=\frac{Bz^{\alpha-1}(z-1)^{\beta-1}}{\prod_{k=1}^{I_{1}}(z-a_{k})\prod_{l=1}^{I_{2}}(z-b_{l})}dz=(\sum_{k=1}^{I_{1}}\frac{1}{z-a_{k}}-\frac{I_{1}}{I_{2}}\sum_{l=1}^{I_{2}}\frac{1}{z-b_{l}})dz,
\end{equation}
where $B\in\mathbb{C}\setminus\{0\}$ is a constant, and $a_{1},\ldots,a_{I},b_{1},\ldots,b_{I}\in \mathbb{C}\setminus\{0,1\}$ are distinct complex numbers,
if and only if one of the following conditions holds.\\
(1) $I_{2}=1$;\\
(2) $I_{2}\geq 2,I_{2}\mid I_{1}$ and $I_{1}\geq\alpha$;\\
(3) $I_{2}\geq2$ and $I_{2}\nmid I_{1}$.
\end{proposition}

\begin{proof}
\textbf{Sufficiency}\par
If $I_{2}=1$, a direct calculation verifies the result.\par
For $I_{2}\geq 2,I_{2}\mid I_{1}$ and $I_{1}\geq\alpha$, by \textbf{Theorem} \ref{M-F-T-3}, there exists a meromorphic function $f:\mathbb{C}\cup\{\infty\}\rightarrow \mathbb{C}\cup\{\infty\}$ of degree $I_{1}$ with branch data $$\Lambda=\{(\underbrace{\frac{I_{1}}{I_{2}},\ldots,\frac{I_{1}}{I_{2}}}_{I_{2}}),(\alpha,1,\ldots,1),(\beta,1,\ldots,1)\}.$$
Without loss of generality, we can express $f$ as
$$f(z)=\frac{C\prod_{k=1}^{I_{1}}(z-a_{k})}{\prod_{l=1}^{I_{2}}(z-b_{l})^{\frac{I_{1}}{I_{2}}}}$$
where $a_{1},\ldots,a_{I_{1}},b_{1},\ldots,b_{I_{2}}\in \mathbb{C}\setminus\{0\}$ are distinct complex numbers and $C\neq0$ is a constant. Then
$$\frac{df}{f}=\frac{Bz^{\alpha-1}(z-1)^{\beta-1}}{\prod_{k=1}^{I_{1}}(z-a_{k})\prod_{l=1}^{I_{2}}(z-b_{l})}dz=(\sum_{k=1}^{I_{1}}\frac{1}{z-a_{k}}-\frac{I_{1}}{I_{2}}\sum_{l=1}^{I_{2}}\frac{1}{z-b_{l}})dz,$$
where $B\neq0$ is a constant, is the desired meromorphic 1-form.\par
For $I_{2}\geq2$ and $I_{2}\nmid I_{1}$, set $(I_{1},I_{2})=m\geq 1$. Then, by \textbf{Theorem} \ref{M-F-T}, there exists a meromorphic function $f:\mathbb{C}\cup\{\infty\}\rightarrow \mathbb{C}\cup\{\infty\}$ of degree $\frac{I_{1}I_{2}}{m}$ and with branch data $$\Lambda=\{(\underbrace{\frac{I_{2}}{m},\ldots,\frac{I_{2}}{m}}_{I_{1}}),(\underbrace{\frac{I_{1}}{m},\ldots,\frac{I_{1}}{m}}_{I_{2}}),(\alpha,1,\ldots,1),(\beta,1,\ldots,1)\}.$$
Without loss of generality, we can express $f$ as
$$f(z)=\frac{C\prod_{k=1}^{I_{1}}(z-a_{k})^{\frac{I_{2}}{m}}}{\prod_{l=1}^{I_{2}}(z-b_{l})^{\frac{I_{1}}{m}}}$$
where $a_{1},\ldots,a_{I_{1}},b_{1},\ldots,b_{I_{2}}\in \mathbb{C}\setminus\{0\}$ are distinct complex numbers and $C\neq0$ is a constant. Then
$$\frac{df}{f}=\frac{Bz^{\alpha-1}(z-1)^{\beta-1}}{\prod_{k=1}^{I_{1}}(z-a_{k})\prod_{l=1}^{I_{2}}(z-b_{l})}dz=\frac{I_{2}}{m}(\sum_{k=1}^{I_{1}}\frac{1}{z-a_{k}}-\frac{I_{1}}{I_{2}}\sum_{l=1}^{I_{2}}\frac{1}{z-b_{l}})dz,$$
where $B\neq0$ is a constant. Thus $\frac{m}{I_{2}}\cdot\frac{df}{f}$ is the desired meromorphic 1-form.\par
\textbf{Necessity}\par
Suppose there exists a meromorphic 1-form on $S^{2}=\mathbb{C}\cup\{\infty\}$ defined by the form
 $$\omega=\frac{Bz^{\alpha-1}(z-1)^{\beta-1}}{\prod_{k=1}^{I_{1}}(z-a_{k})\prod_{l=1}^{I_{2}}(z-b_{l})}dz=(\sum_{k=1}^{I_{1}}\frac{1}{z-a_{k}}-\frac{I_{1}}{I_{2}}\sum_{l=1}^{I_{2}}\frac{1}{z-b_{l}})dz,$$
where $B\in\mathbb{C}\setminus\{0\}$ is a constant and $a_{1},\ldots,a_{I_{1}},b_{1},\ldots,b_{I_{2}}\in \mathbb{C}\setminus\{0\}$ are distinct complex numbers.\par

If $I_{2}=1$ or $I_{2}\geq2$ and $I_{2}\nmid I_{1}$, there is no further proof required. Consequently, we only need to consider the case (2).\par
Suppose $I_{2}\geq2$ and $I_{2}|I_{1}$. Then
$$f(z)=exp(\int \omega)=\frac{C\prod_{k=1}^{I_{1}}(z-a_{k})}{\prod_{l=1}^{I_{2}}(z-b_{l})^{\frac{I_{1}}{I_{2}}}}$$
is a meromorphic function on $\mathbb{C}\cup\{\infty\}$ with degree $I_{1}$, where $C\neq0$ is a constant. The derivative of $f$ is
$$f'(z)=\frac{CBz^{\alpha-1}(z-1)^{\beta-1}}{\prod_{l=1}^{I_{2}}(z-b_{l})^{\frac{I_{1}}{I_{2}}+1}}.$$
Then $\alpha\leq I_{1}$. \par
The proofs of other two cases are similar.
\end{proof}

By \textbf{Proposition} \ref{S-2-P-1} and \textbf{Theorem} \ref{from 1-form to HCMU}, there  exists a non-CSC HCMU metric $g$ on $S^{2}_{\{\alpha,\beta\}}$ such that the singularity are all saddle points of the  Gaussian curvature $K$ of $g$. Furthermore, $K$ has $I_{1}$ maximum points and $I_{2}$ minimum points.

(B) If one singularity is the saddle point of the  Gaussian curvature $K$ of $g$ and another singularity is an extremal point of $K$, it leads to two distinct scenarios.\par
(B-1) Another singularity is a maximal point of $K$.\par
 Denote the maximum and minimum of $K$ by $K_{1}$ and $K_{2}$ respectively. Set
 $$\sigma=-\frac{1}{(K_{1}-K_{2})(K_{2}+2K_{1})},~~~\lambda=-\frac{2K_{1}+K_{2}}{2K_{2}+K_{1}}.$$
 Suppose the character 1-form of $g$ is $\omega$. Without loss of generality, suppose the singularity of singular angle $2\pi\alpha$ is the saddle point of $K$. Since $g$ has two singularities and one is the saddle point of $K$ and another is a maximum point of $K$, $\omega$ has a zero of order $\alpha-1$, $\alpha$ simple poles which are smooth extremal points of $K$ and a simple pole which is the singularity of conical angle $2\pi\beta$.  Suppose $K$ has $I_{1}+1$ maximum points and $I_{2}$ minimum points, then $I_{1}+I_{2}=\alpha$. Regard $S^{2}$ as $\mathbb{C}\cup\{\infty\}$. We can assume that $0$ is the zero of $\omega$, $a_{1},\ldots,a_{I_{1}},a_{I_{1}+1}$ are poles of $\omega$ at which residues are negative and  $b_{1},\ldots,b_{I_{2}}$ are poles of $\omega$ at which residues are positive. Then $Res_{a_{k}}(\omega)=\sigma,k=1,\ldots,I_{1},Res_{a_{I_{1}+1}}(\omega)=\sigma\beta~ \text{and}~ Res_{b_{l}}(\omega)=\sigma\lambda,l=1,\ldots,I_{2}$. This leads to $I_{1}\sigma+\sigma\beta+I_{2}\sigma\lambda=\sigma(I_{1}+\beta+I_{2}\lambda)=0$, which implies $\lambda=-\frac{I_{1}+\beta}{I_{2}}$. Since $\lambda=-\frac{2K_{1}+K_{2}}{2K_{2}+K_{1}}=-\frac{I_{1}+\beta}{I_{2}}<-1$, we obtain $I_{1}+\beta> I_{2}$. Consequently,  there is a nonzero $\widetilde{B}$ such that
 $$\omega=\frac{\widetilde{B}z^{\alpha-1}}{(z-a_{I_{1}+1})\prod_{k=1}^{I_{1}}(z-a_{k})\prod_{l=1}^{I_{2}}(z-b_{l})}dz=\sigma(\sum_{k=1}^{I_{1}}\frac{1}{z-a_{k}}+\frac{\beta}{z-a_{I_{1}+1}}-\frac{I_{1}+\beta}{I_{2}}\sum_{l=1}^{I_{2}}\frac{1}{z-b_{l}})dz.$$

\begin{proposition}\label{S-2-P-2}
Let $\alpha,\beta\geq 2,I_{1}\geq0$ and $I_{2}\geq1$ be $4$ integers such that $I_{1}+I_{2}=\alpha$ and $I_{1}+\beta>I_{2}$. Then, there exists a meromorphic 1-form $\omega$ on $S^{2}=\mathbb{C}\cup\{\infty\}$ defined by the form
\begin{equation}\label{S-2-E0}
\omega=\frac{Bz^{\alpha-1}}{(z-a_{I_{1}+1})\prod_{k=1}^{I_{1}}(z-a_{k})\prod_{l=1}^{I_{2}}(z-b_{l})}dz=(\sum_{k=1}^{I_{1}}\frac{1}{z-a_{k}}+\frac{\beta}{z-a_{I_{1}+1}}-\frac{I_{1}+\beta}{I_{2}}\sum_{l=1}^{I_{2}}\frac{1}{z-b_{l}})dz
\end{equation}
where $B\in\mathbb{C}\setminus\{0\}$ is a constant and $a_{1},\ldots,a_{I_{1}},a_{I_{1}+1},b_{1},\ldots,b_{I_{2}}\in \mathbb{C}\setminus\{0,1\}$ are distinct complex numbers, if and only if one of the following conditions holds.\\
(1) $I_{2}=1$;\\
(2) $I_{2}\geq 2,I_{2}\mid (I_{1}+\beta)$ and $I_{1}+\beta>\alpha-1$;\\
(3) $I_{2}\geq2$ and $I_{2}\nmid (I_{1}+\beta)$.
\end{proposition}

\begin{proof}
\textbf{Sufficiency}\par
If $I_{2}=1$, a direct calculation verifies the result.\par
For $I_{2}\geq 2,I_{2}\mid (I_{1}+\beta)$ and $I_{1}+\beta>\alpha-1$, by the \textbf{Theorem} \ref{M-F-T-3}, there exists a meromorphic function $f:\mathbb{C}\cup\{\infty\}\rightarrow \mathbb{C}\cup\{\infty\}$ of degree $I_{1}+\beta$ and with branch data $$\Lambda=\{(\underbrace{\frac{I_{1}+\beta}{I_{2}},\ldots,\frac{I_{1}+\beta}{I_{2}}}_{I_{2}}),(\beta,\underbrace{1,\ldots,1}_{I_{1}}),(\alpha,1,\ldots,1)\}.$$
Without loss of generality, we can express $f$ as
$$f(z)=\frac{C(z-a_{I_{1}+1})^{\beta}\prod_{k=1}^{I_{1}}(z-a_{k})}{\prod_{l=1}^{I_{2}}(z-b_{l})^{\frac{I_{1}+\beta}{I_{2}}}}$$
where $a_{1},\ldots,a_{I_{1}},a_{I_{1}+1},b_{1},\ldots,b_{I_{2}}\in \mathbb{C}\setminus\{0\}$ are distinct complex numbers and $C\neq0$ is a constant. Then
$$\frac{df}{f}=\frac{Bz^{\alpha-1}}{\prod_{k=1}^{I_{1}}(z-a_{k})\prod_{l=1}^{I_{2}}(z-b_{l})}dz=(\sum_{k=1}^{I_{1}}\frac{1}{z-a_{k}}+\frac{\beta}{z-a_{I_{1}+1}}-\frac{I_{1}+\beta}{I_{2}}\sum_{l=1}^{I_{2}}\frac{1}{z-b_{l}})dz,$$
where $B\neq0$ is a constant, is the desired meromorphic 1-form.\par
For $I_{2}\geq2$ and $I_{2}\nmid (I_{1}+\beta)$, set $(I_{1}+\beta,I_{2})=m\geq 1$, then $\frac{(I_{1}+\beta)I_{2}}{m}\geq2(I_{1}+\beta)>I_{1}+I_{2}=\alpha$. By \textbf{Theorem} \ref{M-F-T-3}, there exists a meromorphic function $f:\mathbb{C}\cup\{\infty\}\rightarrow \mathbb{C}\cup\{\infty\}$ of degree $\frac{(I_{1}+\beta)I_{2}}{m}$ and with branch data $$\Lambda=\{(\frac{\beta I_{2}}{m},\underbrace{\frac{I_{2}}{m},\ldots,\frac{I_{2}}{m}}_{I_{1}}),(\underbrace{\frac{I_{1}+\beta}{m},\ldots,\frac{I_{1}+\beta}{m}}_{I_{2}}),(\alpha,1,\ldots,1)\}.$$
Without loss of generality, we can express $f$ as
$$f(z)=\frac{C(z-a_{I_{1}+1})^{\frac{\beta I_{2}}{m}}\prod_{k=1}^{I_{1}}(z-a_{k})^{\frac{I_{2}}{m}}}{\prod_{l=1}^{I_{2}}(z-b_{l})^{\frac{I_{1}+\beta}{m}}}$$
where $a_{1},\ldots,a_{I_{1}},a_{I_{1}+1},b_{1},\ldots,b_{I_{2}}\in \mathbb{C}\setminus\{0\}$ are disitnct complex numbers and $C\neq0$ is a constant. Then
$$\frac{df}{f}=\frac{Bz^{\alpha-1}}{(z-a_{I_{1}+1})\prod_{k=1}^{I_{1}}(z-a_{k})\prod_{l=1}^{I_{2}}(z-b_{l})}dz=\frac{I_{2}}{m}(\sum_{k=1}^{I_{1}}\frac{1}{z-a_{k}}+\frac{\beta}{z-a_{I_{2}+1}}-\frac{I_{1}+\beta}{I_{2}}\sum_{l=1}^{I_{2}}\frac{1}{z-b_{l}})dz,$$
where $B\neq0$ is a constant. Thus $\frac{m}{I_{2}}\cdot\frac{df}{f}$ is the desired meromorphic 1-form.\par
\textbf{Necessity}\par
Suppose there exists a meromorphic 1-form on $S^{2}=\mathbb{C}\cup\{\infty\}$ defined by the form
 $$\omega=\frac{Bz^{\alpha-1}}{(z-a_{I_{1}+1})\prod_{k=1}^{I_{1}}(z-a_{k})\prod_{l=1}^{I_{2}}(z-b_{l})}dz=(\sum_{k=1}^{I_{1}}\frac{1}{z-a_{k}}+\frac{\beta}{z-a_{I_{1}+1}}-\frac{I_{1}+\beta}{I_{2}}\sum_{l=1}^{I_{2}}\frac{1}{z-b_{l}})dz$$
where $a_{1},\ldots,a_{I_{1}},b_{1},\ldots,b_{I_{2}}\in \mathbb{C}\setminus\{0\}$ are distinct complex numbers and $B\in\mathbb{C}\setminus\{0\}$ is a constant.\par
If $I_{2}=1$ or $I_{2}\geq2$ and $I_{2}\nmid (I_{1}+\beta)$, there is no further proof required. Consequently, we only need to consider the case (2).\par
Suppose $I_{2}\geq2$ and $I_{2}|(I_{1}+\beta)$. Then
$$f(z)=exp(\int \omega)=\frac{C(z-a_{I_{1}+1})^{\beta}\prod_{k=1}^{I_{1}}(z-a_{k})}{\prod_{l=1}^{I_{2}}(z-b_{l})^{\frac{I_{1}+\beta}{I_{2}}}}$$
is a meromorphic function on $\mathbb{C}\cup\{\infty\}$ with degree $I_{1}+\beta$, where $C\neq0$ is a constant. The derivative of $f$ is
$$f'(z)=\frac{CBz^{\alpha-1}(z-a_{I_{1}+1})^{\beta-1}}{\prod_{l=1}^{I_{2}}(z-b_{l})^{\frac{I_{1}}{I_{2}}+1}}.$$
Then $I_{1}+\beta>\alpha-1$.
\end{proof}

By \textbf{Proposition} \ref{S-2-P-2} and \textbf{Theorem} \ref{from 1-form to HCMU}, there  exists a non-CSC HCMU metric $g$ on $S^{2}_{\{\alpha,\beta\}}$ such that the singularity of singular angle $2\pi\alpha$ is the saddle points of the  Gaussian curvature $K$ and the singularity of singular angle $2\pi\beta$ is a maximum point of $K$.

(B-2) Another singularity is a minimal point of $K$.\par
 Denote the maximum and minimum of $K$ by $K_{1}$ and $K_{2}$ respectively. Set
 $$\sigma=-\frac{1}{(K_{1}-K_{2})(K_{2}+2K_{1})},~~~\lambda=-\frac{2K_{1}+K_{2}}{2K_{2}+K_{1}}.$$
 Suppose the character 1-form of $g$ is $\omega$.  Without loss of generality, suppose the singularity of singular angle $2\pi\alpha$ is the saddle point of $K$. Since $g$ has two singularities and one is the saddle point of $K$ and another is a maximum point of $K$, $\omega$ has a zero of order $\alpha-1$, $\alpha$ simple poles which are smooth extremal points of $K$ and a simple pole which is the singularity of conical angle $2\pi\beta$.  Suppose $K$ has $I_{1}$ maximum points and $I_{2}+1$ minimum points, then $I_{1}+I_{2}=\alpha$.  Regard $S^{2}$ as $\mathbb{C}\cup\{\infty\}$. We can assume that $0$ is the zero of $\omega$, $a_{1},\ldots,a_{I_{1}}$ are poles of $\omega$ at which residues are negative and  $b_{1},\ldots,b_{I_{2}},b_{I_{2}+1}$ are poles of $\omega$ at which residues are positive. Then $Res_{a_{k}}(\omega)=\sigma,k=1,\ldots,I_{1},Res_{b_{I_{2}+1}}(\omega)=\sigma\lambda\beta~ \text{and}~ Res_{b_{l}}(\omega)=\sigma\lambda,l=1,\ldots,I_{2}$. From these, we derive $I_{1}\sigma+\sigma\lambda\beta+I_{2}\sigma\lambda=\sigma(I_{1}+\lambda\beta+I_{2}\lambda)=0$, which implies $\lambda=-\frac{I_{1}}{I_{2}+\beta}$. Since $\lambda=-\frac{2K_{1}+K_{2}}{2K_{2}+K_{1}}=-\frac{I_{1}}{I_{2}+\beta}<-1$, we obtain $I_{1}> I_{2}+\beta$. Thus there is a nonzero $\widetilde{B}$ such that
 $$\omega=\frac{\widetilde{B}z^{\alpha-1}}{\prod_{k=1}^{I_{1}}(z-a_{k})\prod_{l=1}^{I_{2}+1}(z-b_{l})}dz=\sigma(\sum_{k=1}^{I_{1}}\frac{1}{z-a_{k}}-\frac{\frac{\beta I_{1}}{I_{2}+\beta}}{z-b_{I_{2}+1}}-\frac{I_{1}}{I_{2}+\beta}\sum_{l=1}^{I_{2}}\frac{1}{z-b_{l}})dz.$$

\begin{proposition}\label{S-2-P-3}
Let $\alpha,\beta\geq 2,I_{1}\geq1$ and $I_{2}\geq0$ be integers such that $I_{1}+I_{2}=\alpha$ and $I_{1}>I_{2}+\beta$. Then, there exists a meromorphic 1-form $\omega$ on $S^{2}=\mathbb{C}\cup\{\infty\}$ defined as follows
\begin{equation}\label{S-2-E0}
\omega=\frac{Bz^{\alpha-1}}{\prod_{k=1}^{I_{1}}(z-a_{k})\prod_{l=1}^{I_{2}+1}(z-b_{l})}dz=(\sum_{k=1}^{I_{1}}\frac{1}{z-a_{k}}-\frac{\frac{\beta I_{1}}{I_{2}+\beta}}{z-b_{I_{2}+1}}-\frac{I_{1}}{I_{2}+\beta}\sum_{l=1}^{I_{2}}\frac{1}{z-b_{l}})dz
\end{equation}
where $B\in\mathbb{C}\setminus\{0\}$ is a constant and $a_{1},\ldots,a_{I_{1}},b_{1},\ldots,b_{I_{2}},b_{I_{2}+1}\in \mathbb{C}\setminus\{0,1\}$ are distinct complex numbers, if and only if one of the following conditions holds.\\
(1) $I_{2}=0$;\\
(2) $I_{2}\geq1,(I_{2}+\beta)\nmid I_{1}$ and $(I_{2}+\beta)I_{1} >(\alpha-1){\rm GCD}(I_{2}+\beta,I_{1})$.
\end{proposition}

\begin{proof}
\textbf{Sufficiency}\par
If $I_{2}=0$,  a direct calculation verifies the result.\par
For $I_{2}\geq1,(I_{2}+\beta)\nmid I_{1}$ and $(I_{2}+\beta)I_{1} >(\alpha-1){\rm GCD}(I_{2}+\beta,I_{1})$, set $m={\rm GCD}(I_{2}+\beta,I_{1})$. Then, by \textbf{Theorem} \ref{M-F-T-3}, there exists a meromorphic function $f:\mathbb{C}\cup\{\infty\}\rightarrow \mathbb{C}\cup\{\infty\}$ of degree $\frac{I_{1}(I_{2}+\beta)}{m}$ and with branch data $$\Lambda=\{(\underbrace{\frac{I_{2}+\beta}{m},\ldots,\frac{I_{2}+\beta}{m}}_{I_{1}}),(\frac{I_{1}\beta}{m},\underbrace{\frac{I_{1}}{m},\ldots,\frac{I_{1}}{m}}_{I_{2}}),(\alpha,1,\ldots,1)\}.$$
Without loss of generality, we can express $f$ as
$$f(z)=\frac{C\prod_{k=1}^{I_{1}}(z-a_{k})^{\frac{I_{2}+\beta}{m}}}{(z-b_{I_{2}+1})^{\frac{I_{1}\beta}{m}}\prod_{l=1}^{I_{2}}(z-b_{l})^{\frac{I_{1}}{m}}}$$
where $a_{1},\ldots,a_{I_{1}},b_{1},\ldots,b_{I_{2}+1}\in \mathbb{C}\setminus\{0\}$ are distinct complex numbers and $C\neq0$ is a constant. Then, a direct calculation show that
$$\frac{m}{I_{2}+\beta}\cdot\frac{df}{f}=\frac{Bz^{\alpha-1}}{\prod_{k=1}^{I_{1}}(z-a_{k})\prod_{l=1}^{I_{2}+1}(z-b_{l})}dz=(\sum_{k=1}^{I_{1}}\frac{1}{z-a_{k}}-\frac{\frac{\beta I_{1}}{I_{2}+\beta}}{z-b_{I_{2}+1}}-\frac{I_{1}}{I_{2}+\beta}\sum_{l=1}^{I_{2}}\frac{1}{z-b_{l}})dz,$$
where $B\neq0$ is a constants, is the desired meromorphic 1-form.\par
\textbf{Necessity}\par
Suppose there exists a meromorphic 1-form on $S^{2}=\mathbb{C}\cup\{\infty\}$ defined by the form
 $$\omega=\frac{Bz^{\alpha-1}}{\prod_{k=1}^{I_{1}}(z-a_{k})\prod_{l=1}^{I_{2}+1}(z-b_{l})}dz=(\sum_{k=1}^{I_{1}}\frac{1}{z-a_{k}}-\frac{\frac{\beta I_{1}}{I_{2}+\beta}}{z-b_{I_{2}+1}}-\frac{I_{1}}{I_{2}+\beta}\sum_{l=1}^{I_{2}}\frac{1}{z-b_{l}})dz$$
where $a_{1},\ldots,a_{I_{1}},b_{1},\ldots,b_{I_{2}+1}\in \mathbb{C}\setminus\{0\}$ are distinct complex numbers and $B\in\mathbb{C}\setminus\{0\}$ is a constant.\par
If $I_{2}=0$, there is nothing to prove. If $I_{2}\geq1$ and $(I_{2}+\beta)\mid I_{1}$, then
$$f(z)=exp(\int \omega)=\frac{C\prod_{k=1}^{I_{1}}(z-a_{k})}{(z-b_{I_{2}+1})^{\frac{I_{1}\beta}{I_{2}+\beta}}\prod_{l=1}^{I_{2}}(z-b_{l})^{\frac{I_{1}}{I_{2}+\beta}}}$$
is a meromorphic function of degree $I_{2}$ on $\mathbb{C}\cup\{\infty\}$, where $C\neq0$ is a constant. Since $\omega=\frac{df}{f}$, we obtain $I_{1}>\alpha-1$. However, $I_{1}+I_{2}=\alpha$. It is a contradiction.\par

Suppose $I_{2}\geq1$ and $(I_{2}+\beta)\nmid I_{1}$. Set $m={\rm GCD}(I_{2}+\beta,I_{1})$, then
$$f(z)=exp(\int \frac{I_{2}+\beta}{m}\omega)=\frac{C\prod_{k=1}^{I_{1}}(z-a_{k})^{\frac{I_{2}+\beta}{m}}}{(z-b_{I_{2}+1})^{\frac{I_{1}\beta}{m}}\prod_{l=1}^{I_{2}}(z-b_{l})^{\frac{I_{1}}{m}}}$$
is a meromorphic function of degree $\frac{I_{1}(I_{2}+\beta)}{m}$ on $\mathbb{C}\cup\{\infty\}$, where $C\neq0$ is a constant. The derivative of $f$ is
$$f'(z)=\frac{CBz^{\alpha-1}\prod_{k=1}^{I_{1}}(z-a_{k})^{\frac{I_{2}+\beta}{m}-1}}{(z-b_{I_{2}+1})^{\frac{I_{1}\beta}{m}+1}\prod_{l=1}^{I_{2}}(z-b_{l})^{\frac{I_{1}}{m}+1}}.$$
Then $\frac{(I_{2}+\beta)I_{1}}{m} >(\alpha-1)$, i.e., $(I_{2}+\beta)I_{1} >(\alpha-1){\rm GCD}(I_{2}+\beta,I_{1})$.
\end{proof}

By \textbf{Proposition} \ref{S-2-P-3} and \textbf{Theorem} \ref{from 1-form to HCMU}, there exists a non-CSC HCMU metric $g$ on $S^{2}_{\{\alpha,\beta\}}$ such that the singularity of singular angle $2\pi\alpha$ is the saddle point of the  Gaussian curvature $K$ and the singularity of singular angle $2\pi\beta$ is a minimum point of $K$.\par
\section*{Acknowledgments}
~~~~
We would like to express our gratitude to the reviewers for their constructive comments and suggestions, which have helped improve the quality and clarity of this manuscript. Wei expresses his deep gratitude to Professor Yingyi Wu and Professor Bin Xu for their invaluable encouragement and the insightful discussions that greatly contributed to this work.
This work is supported by the National Natural Science Foundation of Henan (Grant No. 202300410067) and partially supported by the National Natural Science Foundation of China (No. 12171140).

\textbf{Declarations}

\textbf{Data Availability Statement}  This manuscript has no associated data.

\textbf{Competing interests} On behalf of all authors, the corresponding author states that there is no conflict of interest.


\par\vskip0.3cm
\noindent
Yingjie Meng\\
School of Mathematics and Statistics \\
Henan University  \\
Kaifeng 475004 P.R. China \\
mengyingjie@henu.edu.cn\\
\\
Zhiqiang Wei\\
School of Mathematics and Statistics \\
Henan University  \\
Kaifeng 475004 P.R. China \\
Center for Applied Mathematics of Henan Province\\
Henan University\\
Zhengzhou 450046 P.R. China\\
10100123@vip.henu.edu.cn\\
weizhiqiang15@mails.ucas.ac.cn\\
\end{document}